\documentclass[a4paper,12pt,draft]{amsart}
\usepackage[margin=30mm]{geometry}
\usepackage{amssymb, amsmath, amsthm, amscd}

\usepackage[T1]{fontenc}
\usepackage{eucal,mathrsfs,dsfont}
\usepackage{color}

\renewcommand{\le}{\leqslant}
\renewcommand{\ge}{\geqslant}

\definecolor{mno}{rgb}{0.5,0.1,0.5}

\newcommand{\R}{\mathds R}

\newcommand{\var}{\textup{Var}}

\newtheorem{theorem}{Theorem}[section]
\newtheorem{lemma}[theorem]{Lemma}

\theoremstyle{definition}

\newtheorem{example}[theorem]{Example}

\begin{document}
\allowdisplaybreaks
\title[Functional Inequalities for Non-local Dirichlet Forms] {\bfseries A Simple Approach to Functional Inequalities for Non-local Dirichlet Forms}

\author{Jian Wang}
\thanks{\emph{J.\ Wang:}
   School of Mathematics and Computer Science, Fujian Normal University, 350007 Fuzhou, P.R. China. \texttt{jianwang@fjnu.edu.cn}
   }

\date{}

\maketitle

\begin{abstract} With direct and simple proofs,
we establish Poincar\'{e} type inequalities
(including Poincar\'{e} inequalities, weak Poincar\'{e} inequalities
and super Poincar\'{e} inequalities), entropy inequalities and
Beckner-type inequalities for non-local Dirichlet forms. The proofs
are efficient for non-local Dirichlet forms with general jump
kernel, and also work for $L^p (p>1)$ settings. Our results yield a new sufficient condition for fractional Poincar\'{e} inequalities, which were recently studied in \cite{MRS,Gre}. To our knowledge
this is the first result providing entropy inequalities and Beckner-type
inequalities for measures more general than L\'{e}vy measures.   

\medskip


\medskip

\end{abstract}

\section{Introduction and Main Results}
The question of obtaining Poincar\'{e}-type inequalities (or more generally entropy inequalities) for pure jump L\'{e}vy processes was studied in the last decades, e.g.\ see \cite{BL,Wu,Cha}.
In particular, it was proved by \cite[Corollary 4.2]{Wu} and \cite[Theorem 23]{Cha}
that
\begin{equation}\label{phiineq}\textrm{Ent}_\mu^\Phi(f)\le \iint D_\Phi(f(x),f(x+z))\,\nu_\mu(dz)\,\mu(dx),\quad f\in C_b^\infty(\R^d), f>0\end{equation}
with
$$\textrm{Ent}_\mu^\Phi(f)=\int \Phi (f)\, d\mu-\Phi\Big(\int f d\mu\Big)$$ and $D_\Phi$ is the so-called Bergman distance associated with $\Phi$:
$$D_\Phi(a,b)=\Phi(a)-\Phi(b)-\Phi'(b)(a-b),$$
where $\mu$ is a rather
general probability measure and $\nu_\mu$ is the (singular) L\'{e}vy
measure associated to $\mu$. By setting $\Phi(x)=x^2$ and $\Phi(x)=x\log x$, $\textrm{Ent}_\mu^\Phi(f)$ becomes the classical variance $\var_\mu(f)$ and entropy $ \textrm{Ent}_\mu(f)$ respectively, and so \eqref{phiineq} yields the Poincar\'{e} inequality and the entropy inequality for the choice of measure $(\mu,\nu_\mu)$. Note that either one of the measures $\mu$ and $\nu_\mu$ in \eqref{phiineq} uniquely specifies the other, and so this is  a
strong constraint to study functional inequalities for general
non-local Dirichlet forms. The first breakthrough in this direction
was established in \cite{MRS} by using the methods from harmonic
analysis, and then was extended in \cite{Gre} to $L^p$ weighted Poincar\'{e} inequalities and generalized logarithmic Sobolev inequalities in an abstract situation.

Let $V$ be a locally bounded measurable function on $\R^d$ such that $\int e^{-V(x)}\,dx=1$; that is, $\mu_V(dx):=e^{-V(x)}\,dx$ is a probability measure on $\R^d$.  The main
result in \cite{MRS} (see \cite[Theorem 1.2]{MRS})
 states that, if $V\in C^2(\R^d)$ such that for some constant
$\varepsilon>0$,
\begin{equation*}\frac{(1-\varepsilon)|\nabla V(x)|^2}{2} -\Delta V(x)\to\infty,\qquad
|x|\to \infty,\end{equation*} then there exist two positive constants
$\delta$ and $C_0$ such that for all $f\in C_b^\infty(\R^d)$,
$$\aligned\int (f-\mu_V(f))^2\big(1+|\nabla  V|^{\alpha}\big) &\,\mu_V(dx)
\leqslant C_0D_{\alpha,V,\delta}(f,f),\endaligned$$
where
\begin{equation}\label{MRS}D_{\alpha,V,\delta}(f,f)=\iint
\frac{(f(y)-f(x))^2}{|y-x|^{d+\alpha}}e^{-\delta|y-x|}\,dy
\,\mu_V(dx).\end{equation}
According to the paragraph below
\cite[Remark 1.3]{MRS}, \eqref{MRS} is natural in the sense that:
 we should regard the measure
${|y-x|^{-(d+\alpha)}}e^{-\delta|y-x|}\,dy$ as the L\'{e}vy measure, and $\mu_V(dx)$
as the ambient measure. Namely, $D_{\alpha,V,\delta}$ does get rid of the
constraint of $(\mu,\nu_\mu)$ in \eqref{phiineq}, and it should be a typical example in
study functional inequalities for non-local Dirichlet forms. This leds us to consider the Dirichlet form $( D_{\rho, V},\mathscr{ D}(D_{\rho, V}))$ as follows.
Let $\rho$ be a strictly positive measurable function on $(0,\infty)$ such that $\int_{(0,\infty)}\rho(r)(1\wedge r^2)r^{d-1}\,dr<\infty.$
Let $L^2(\mu_V)$ be the space of Borel measurable functions $f$ on $\R^d$ such that $\mu_V(f^2):=\int f^2(x)\,\mu_V(dx)<\infty$. Set
$$\aligned  D_{\rho, V}(f,f):=&\iint_{x\neq y}\big(f(x)-f(y)\big)^2\rho(|x-y|)\,dy\,\mu_V(dx)\\
\mathscr{D}(D_{\rho,V}):=&\bigg\{ f\in L^2(\mu_V): D_{\rho,V}(f,f)<\infty\bigg\}.\endaligned$$
According to \cite[Example 2.2]{CW}, we know that $( D_{\rho, V},\mathscr{ D}(D_{\rho, V}))$ is a symmetric Dirichlet form such that $C_b^\infty(\R^d)\subset \mathscr{D}(D_{\rho, V}) $, where $C_b^\infty(\R^d)$ denotes the set of smooth functions on $\R^d$ with bounded derivatives for all orders.

The purpose of this note is to present sufficient conditions for Poincar\'{e} type inequalities (i.e.\ Poincar\'{e} inequalities, weak Poincar\'{e} inequalities and super Poincar\'{e} inequalities), entropy inequalities and Beckner-type inequalities for $( D_{\rho, V},\mathscr{ D}(D_{\rho, V}))$. We first state the main result for Poincar\'{e} type inequalities of $( D_{\rho, V},\mathscr{ D}(D_{\rho, V}))$.

\begin{theorem}\label{th1} $(1)$ If there exists a constant $c>0$ such that for any $x$, $y\in \R^d$ with $x\neq y$,
\begin{equation}\label{th1.1}\big(e^{V(x)}+e^{V(y)}\big)\rho(|x-y|)\ge c,\end{equation} then the following Poincar\'{e} inequality
\begin{equation}\label{th1.1.1}\mu_V(f-\mu_V(f))^2 \le c^{-1} D_{\rho,V}(f,f), \quad f\in \mathscr{ D}(D_{\rho, V})\end{equation}  holds.

$(2)$ For any probability measure $\mu_V$, the following weak Poincar\'{e} inequality \begin{equation}\label{th1.2.1}\mu_V(f-\mu_V(f))^2\le \alpha(r)D_{\rho,V}(f,f)+r\|f\|_\infty^2,\quad r>0, f\in \mathscr{ D}(D_{\rho, V})\end{equation}
holds with  $$\alpha(r)=\inf\Bigg\{\frac{1}{\inf\limits_{0<|x-y|\le s}\big[(e^{V(x)}+e^{V(y)})\rho(|x-y|)\big]}: \iint_{|x-y|> s}\mu_V(dy)\,\mu_V(dx)\le \frac{r}{2}\Bigg\}.$$

$(3)$ Suppose that
there exists a nonnegative locally bounded measurable function $w$ on $\R^d$ such that $$\lim\limits_{|x|\to\infty}w(x)=\infty,$$ and for any $x$, $y\in \R^d$ with $x\neq y$,
\begin{equation}\label{th1.2} e^{V(x)}+e^{V(y)}\ge\frac{ w(x)+w(y)}{\rho(|x-y|)}.\end{equation} Then the following super Poincar\'{e} inequality
\begin{equation}\label{th1.2.2}\mu_V(f^2)\le r D_{\rho,V}(f,f)+ \beta(r)\mu_V(|f|)^2,\quad r>0, f\in \mathscr{ D}(D_{\rho, V}) \end{equation} holds with
$$\beta(r)=\inf\bigg\{\frac{2\mu_V(\omega)}{\inf\limits_{|x|\ge t}\omega(x)}+\beta_t(t\wedge s): \frac{2}{\inf_{|x|\ge t}w(x)}+s\le r\textrm{ and } t,s>0\bigg\},$$ where for any $t>0$,
$$\beta_t(s)=\inf\left\{ \frac{c_0\Big(\sup\limits_{|z|\le 2t}e^{V(z)}\Big)^2}{u^d\Big(\inf\limits_{|z|\le t} e^{V(z)}\Big)}:  \frac{c_0\Big(\sup\limits_{0<\varepsilon\le u}\rho(\varepsilon)^{-1}\Big)\Big(\sup\limits_{|z|\le 2t}e^{V(z)}\Big)}{u^d\Big(\inf\limits_{|z|\le t} e^{V(z)}\Big)}\le s \textrm{ and } u>0 \right\}.$$
 \end{theorem}

\bigskip

To illustrate the power of Theorem \ref{th1}, we will consider the following examples.

\begin{example}\label{ex1} Let $\mu_V(dx)=e^{-V(x)}\,dx:=C_{d,\varepsilon} (1+|x|)^{-(d+\varepsilon)}\,dx$ with $\varepsilon>0$, and $\rho(r)=r^{-(d+\alpha)}$ with $\alpha\in (0,2)$.

\begin{itemize}

\item[(1)] If $\varepsilon\ge \alpha$, then the Poincar\'{e} inequality \eqref{th1.1.1} holds with $c=\frac{2^{1-(d+\alpha)}}{C_{d,\varepsilon}}$.

\item[(2)] If $0<\varepsilon<\alpha$, then the weak Poincar\'{e} inequality \eqref{th1.2.1} holds with $$\alpha(r)=c_1\big(1+ r^{-(\alpha-\varepsilon)/\varepsilon}\big)$$ for some constant $c_1>0$.

\item[(3)] If $\varepsilon>\alpha$, then the super Poincar\'{e} inequality \eqref{th1.2.2} holds with $$\beta(r)=c_2\bigg(1+ r^{-\frac{d}{\alpha}-\frac{(d+\varepsilon)(d+2\alpha)}{\alpha(\varepsilon-\alpha)}}\bigg)$$ for some constant $c_2>0$.

\end{itemize}
According to \cite[Corollary 1.2]{WW}, we know that all the conclusions above are optimal.
\end{example}

\begin{example} Let $\mu_V(dx)=e^{-V(x)}\,dx:=C_{d,\alpha,\varepsilon}(1+|x|)^{-(d+\alpha)}\log^\varepsilon(e+|x|)\,dx$ with $\varepsilon\in\R$, and $\rho(r)=r^{-(d+\alpha)}$ with $\alpha\in (0,2)$. \begin{itemize}

\item[(1)] If $\varepsilon\le0$, then the Poincar\'{e} inequality \eqref{th1.1.1} holds with $c=\frac{2^{1-(d+\alpha)}}{C_{d,\varepsilon}}$.

\item[(2)] If $\varepsilon>0$, then the weak Poincar\'{e} inequality \eqref{th1.2.1} holds with $$\alpha(r)=c_3\big(1+ \log^{\varepsilon}(1+r^{-1})\big)$$ for some constant $c_3>0$.

\item[(3)] If $\varepsilon<0$, then the super Poincar\'{e} inequality \eqref{th1.2.2} holds with $$\beta(r)=\exp\Big(c_4(1+r^{1/\varepsilon})\Big)$$ for some constant $c_4>0$.

\end{itemize}
By \cite[Corollary 1.3]{WW}, all the conclusions above are also sharp.
\end{example}

\begin{example} Let $\mu_V(dx)=e^{-V(x)}\,dx:=C_\lambda e^{-\lambda |x|}\,dx$ with $\lambda>0$, and $\rho(r)=e^{-\delta r}r^{-(d+\alpha)}$ with $\delta\ge0$ and $\alpha\in (0,2)$. Therefore, if $\lambda>2\delta$, then the super Poincar\'{e} inequality \eqref{th1.2.2} holds with $\beta(r)=c_5\Big(1+ r^{-\frac{d}{\alpha}-\frac{2\lambda(d+2\varepsilon)}{\alpha(\lambda-2\delta)}}\Big)$ for some constant $c_5>0$. In particular, the Poincar\'{e} inequality \eqref{th1.1.1} holds. Note that, this conclusion can not be deduced from \cite[Theorem 1.1]{MRS}, see also the statement before \eqref{MRS}.
\end{example}

Next, we turn to study entropy inequalities and Beckner-type inequalities for $( D_{\rho, V},\mathscr{ D}(D_{\rho, V}))$. Recall that for any $f\in \mathscr{ D}(D_{\rho, V})$ with $f>0$,
$$\textrm{Ent}_{\mu_V}(f):=\mu_V(f\log f)-\mu_V(f)\log\mu_V(f).$$

\begin{theorem}\label{th2} Suppose that \eqref{th1.1} is satisfied. Then the following entropy inequality
\begin{equation}\label{th2.2}\textrm{Ent}_{\mu_V}(f) \le c^{-1} D_{\rho,V}(f,\log f)\end{equation}  holds for all $f\in \mathscr{ D}(D_{\rho, V})$ with $f>0$; and moreover, the following Beckner-type inequality also holds: for any $p\in(1,2]$ and $f\in \mathscr{ D}(D_{\rho, V})$ with $f\ge0$,
\begin{equation}\label{th2.3}\mu_V(f^p)-\mu_V(f)^p\le c^{-1} D_{\rho,V}(f,f^{p-1}).\end{equation}
 \end{theorem}

The entropy inequality \eqref{th2.2} and Beckner-type inequality \eqref{th2.3} are stronger than the Poincar\'{e} inequality \eqref{th1.1.1} (To see this, one can apply these inequalities to the function $1+\varepsilon f$ and then take the limit as $\varepsilon\to0$). Clearly, the Beckner-type inequality \eqref{th2.3} reduces to the Poincar\'{e} inequality \eqref{th1.1.1} if $p=2$, whereas dividing both sides by $p-1$ and taking the limit as $p\to1$ we obtain the entropy inequality \eqref{th2.2}.
As mentioned in the remarks below \eqref{MRS}, comparing Theorem \ref{th2} with \eqref{phiineq} the improvement is due to that
we do not impose any link between the measure $\mu_V(dx)$ on $x$ and
the singular measure $\rho(|z|)\,dz$ on $z=y-z.$  This is to our
knowledge the first result that gets rid of the strong constraint
for entropy inequalities  and Beckner-type inequalities of non-local Dirichlet forms.

\begin{example}\label{ex2}[{\bf Continuation of Example \ref{ex1}}] Let $$\mu_V(dx)=e^{-V(x)}\,dx:=C_{d,\varepsilon} (1+|x|)^{-(d+\varepsilon)}\,dx$$ with $\varepsilon>0$, and $\rho(r)=r^{-(d+\alpha)}$ with $\alpha\in (0,2)$.

\begin{itemize}

\item[(1)] If $\varepsilon\ge \alpha$, then, according to Theorem \ref{th2}, the entropy inequality \eqref{th2.2} and Beckner-type inequality \eqref{th2.3} hold with $c=\frac{2^{1-(d+\alpha)}}{C_{d,\varepsilon}}$.

\item[(2)] If $0<\varepsilon<\alpha$, then, according to \cite[Corollary 1.2]{WW}, the Poincar\'{e} inequality \eqref{th1.1.1} does not hold. Hence, by the remark below Theorem \ref{th2}, both the entropy inequality \eqref{th2.2} and Beckner-type inequality \eqref{th2.3} do not hold.
\end{itemize}
According to Examples \ref{ex1}, \ref{ex2} and \cite[Corollary 1.2]{WW}, we know that for the probability measure $\mu_V(dx)=e^{-V(x)}\,dx:=C_{d,\alpha} (1+|x|)^{-(d+\alpha)}\,dx$, it fulfills the entropy inequality \eqref{th2.2} and the Beckner-type inequality \eqref{th2.3}, but not the super Poincar\'{e} inequality \eqref{th1.2.2}.
\end{example}

\section{Proofs of Theorems and Example \ref{ex1}}
\begin{proof}[Proof of Theorem $\ref{th1}$] (1) For any $f\in \mathscr{D}(D_{\rho, V}),$  \begin{equation}\label{proof1}\aligned &\frac{1}{2} \iint \big(f(x)-f(y)\big)^2\,\mu_V(dy)\,\mu_V(dx)\\
&=\frac{1}{2} \iint \big(f^2(x)+f^2(y)-2f(x)f(y)\big)\,\mu_V(dy)\,\mu_V(dx)\\
&=\mu_V(f^2)-\mu_V(f)^2=\mu_V(f-\mu_V(f))^2.\endaligned\end{equation}
On the other hand, by \eqref{th1.1}, we find that
$$\aligned
&\frac{1}{2} \iint \big(f(x)-f(y)\big)^2\,\mu_V(dy)\,\mu_V(dx)\\
&= \frac{1}{2}\iint_{x\neq y} \big(f(x)-f(y)\big)^2\rho(|x-y|)\rho(|x-y|)^{-1}\,\mu_V(dy)\,\mu_V(dx)\\
&\le c^{-1}\iint_{x\neq y} \big(f(x)-f(y)\big)^2\rho(|x-y|)\,\frac{e^{-V(x)}+e^{-V(y)}}{2}\,dy\,dx\\
&= c^{-1} D_{\rho,V}(f,f),
\endaligned$$ which, along with \eqref{proof1}, yields the required assertion.

(2) According to \eqref{proof1}, for any $s>0$ and $f\in \mathscr{D}(D_{\rho, V}),$
\begin{align*} \mu_V(f-\mu_V(f))^2&=\frac{1}{2} \iint (f(x)-f(y))^2\,\mu_V(dy)\,\mu_V(dx)\\
&=\frac{1}{2}\iint_{0<|x-y|\le s} (f(x)-f(y))^2\,\mu_V(dy)\,\mu_V(dx) \\
&\quad +\frac{1}{2}\iint_{|x-y|> s} (f(x)-f(y))^2\,\mu_V(dy)\,\mu_V(dx)\\
&\le \iint_{0<|x-y|\le s} (f(x)-f(y))^2\rho(|x-y|)\\
&\qquad\qquad\times  \bigg[\frac{e^{-V(x)-V(y)}}{(e^{-V(x)}+e^{-V(y)})\rho(|x-y|)}\bigg]\frac{e^{-V(x)}+e^{-V(y)}}{2}\,dy\,dx\\
&\quad+ 2\|f\|_\infty^2 \iint_{|x-y|> s} \,\mu_V(dy)\,\mu_V(dx)\\
&\le \bigg(\sup_{0<|x-y|\le s}\frac{1}{(e^{V(x)}+e^{V(y)})\rho(|x-y|)}\bigg) D_{\rho,V}(f,f)\\
&\quad+\bigg(2 \iint_{|x-y|>s} \,\mu_V(dy)\,\mu_V(dx)\bigg)\|f\|_\infty^2.
\end{align*} The desired assertion follows from the definition of $\alpha$.

(3) For any $f\in \mathscr{D}(D_{\rho, V})$, by Jensen's inequality,
$$\aligned \mu_V((f-\mu_V(f))^2w)&=\int \Big(f(x)-\int f(y)\,\mu_V(dy)\Big)^2w(x)\,\mu_V(dx)\\
&=\int\Big(\int(f(x)-f(y))\,\mu_V(dy)\Big)^2w(x)\,\mu_V(dx)\\
&\le \iint (f(x)-f(y))^2w(x)\,\mu_V(dy)\,\mu_V(dx).
\endaligned$$
This implies that
$$\aligned \mu_V((f-\mu_V(f))^2w)&\le \iint (f(x)-f(y))^2\frac{w(x)+w(y)}{2}\,\mu_V(dy)\,\mu_V(dx).
\endaligned$$ Thus, by \eqref{th1.2}, we arrive at
\begin{equation}\label{prof1}\aligned\mu_V((f-\mu_V(f))^2w)&\le\iint (f(x)-f(y))^2\\
&\qquad\qquad\times \rho(|x-y|)\frac{e^{V(x)}+e^{V(y)}}{2}\,\mu_V(dy)\,\mu_V(dx)\\
&\le D_{\rho,V}(f,f).\endaligned \end{equation}

Next, we will follow the proof of \cite[Proposition 1.6]{CW} to obtain the super Poincar\'{e} inequality from \eqref{prof1}.
We first claim that $\mu_V(\omega)<\infty$. In fact, let $C_c^\infty(\R^d)$ be the set of smooth functions on $\R^d$ with compact support.
Choose a function $g\in C_c^{\infty}(\R^d)$ such that $g(x)=0$ for every $|x|\ge 1$ and
$\mu_V(g)=1$. Then, applying this test function $g$ into (\ref{prof1}) and noting the fact that $C_c^\infty(\R^d)\subset C_b^\infty(\R^d)\subset \mathscr{D}(D_{\rho, V})$, we have
\begin{equation*}
\begin{split}
\int_{\{|x|\ge 1\}}\omega(x)\,\mu_V(dx)&\le \int \big(g(x)-\mu_V(g)\big)^2\omega(x)\,\mu_V(dx)\leqslant D_{\rho,V}(g,g)<\infty.
\end{split}
\end{equation*}
Since the function $\omega$ is bounded on $\{x \in \R^d: |x|\le 1\}$,
$\int_{\{|x|\le 1\}}\omega(x)\,\mu_V(dx)<\infty$. Combining both estimates above, we prove the desired claim.

For any $t>1$ large enough and $f \in \mathscr{D}(D_{\rho, V})$, by (\ref{prof1}), we have
\begin{equation*}
\begin{split}
\int_{\{|x|\ge t\}}f^2(x)\,\mu_V(dx)&\le \frac{1}{\inf\limits_{|x|\ge t}\omega(x)}\int f^2(x)\omega(x)\, \mu_V(dx)\\
&\le \frac{2}{\inf\limits_{|x|\ge t}\omega(x)}\int \big(f(x)-\mu_V(f)\big)^2\omega(x)\, \mu_V(dx)\\
&\quad+ \frac{2}{\inf\limits_{|x|\ge t}\,\omega(x)}\int \mu_V(f)^2\omega(x)\, \mu_V(dx)\\
&\le \frac{2}{\inf\limits_{|x|\ge t}\omega(x)}\bigg(D_{\rho,V}(f,f)+
{\mu_V(\omega)}\,\mu_V(|f|)^2\bigg),
\end{split}
\end{equation*} where the second inequality follows from the inequality that for any $a$, $b\in\R$, $a^2\le 2(a-b)^2+2b^2.$

 On the other hand, Lemma \ref{lemma} below shows that the local super
Poincar\'{e}  inequality
\begin{equation}\label{local}
\begin{split}
& \int_{\{|x|\le t\}}f^2(x)\,\mu_V(dx)\le sD_{\rho,V}(f,f)+ \beta_t(t\wedge s)\mu_V(|f|)^2,\quad s>0
\end{split}
\end{equation}
holds for any $t
 > 1$ and $f\in \mathscr{D}(D_{\rho, V})$.

Combining both estimates above, we get that for $t>1$ large enough and any $f\in \mathscr{D}(D_{\rho, V})$,
\begin{equation*}
\begin{split}
\mu_V(f^2)\le& \Big(\frac{2}{\inf\limits_{|x|\ge t}\omega(x)}+s\Big)D_{\rho,V}(f,f)+\bigg(\frac{2\mu_V(\omega)}{\inf\limits_{|x|\ge t}\omega(x)}+\beta_t(t\wedge s)\bigg)\mu_V(|f|)^2,\quad  s>0.
\end{split}
\end{equation*} This, along with $\lim\limits_{|x|\rightarrow \infty}\omega(x)=\infty$  and the definition of $\beta$, yields the required super Poincar\'{e} inequality.
 \end{proof}

For the local super Poincar\'{e} inequality \eqref{local} in part (3) of the proof above, we can see from the following

\begin{lemma}\label{lemma} For any $f\in \mathscr{D}(D_{\rho, V})$ and $r>0$, we have
$$ \int_{B(0,r)}f^2(x)\,\mu_V(dx)\le s D_{\rho, V}(f,f)+\beta_r(r\wedge s) \mu_V(|f|)^2\quad s>0,$$ where
$$\beta_r(s)=\inf\left\{ \frac{2\Big(\sup\limits_{|z|\le 2r}e^{V(z)}\Big)^2}{|B(0,t)|\Big(\inf\limits_{|z|\le r} e^{V(z)}\Big)}:  \frac{2\Big(\sup\limits_{0<\varepsilon\le t}\rho(\varepsilon)^{-1}\Big) \Big(\sup\limits_{|z|\le 2r}e^{V(z)}\Big)}{|B(0,t)|\Big(\inf\limits_{|z|\le r} e^{V(z)}\Big)}\le s\textrm{ and }t>0\right\},$$ and $|B(0,t)|$ denotes the volume of the ball with radius $t$.
\end{lemma}
\begin{proof} (1) For any $0<s\le r$ and $f\in \mathscr{D}(D_{\rho, V})$, define
$$f_s(x):=\frac{1}{|B(0,s)|}\int_{B(x,s)}f(z)\,dz,\quad x\in B(0,r).$$ We have
$$\sup_{x\in B(0,r)}|f_s(x)|\le \frac{1}{|B(0,s)|} \int_{B(0,2r)}|f(z)|\,dz,$$
and $$\aligned \int_{B(0,r)}|f_s(x)|\,dx&\le \int_{B(0,r)}\frac{1}{|B(0,s)|}\int_{B(x,s)}|f(z)|\,dz\,dx\\
&\le \int_{B(0,2r)}\bigg(\frac{1}{|B(0,s)|}\int_{B(z,s)}\,dx\bigg)|f(z)|\,dz\le \int_{B(0,2r)}|f(z)|\,dz.
\endaligned$$ Thus,
$$\aligned\int_{B(0,r)}f_s^2(x)\,dx\le & \Big(\sup_{x\in B(0,r)}|f_s(x)|\Big) \int_{B(0,r)}|f_s(x)|\,dx\\
\le &\frac{1}{|B(0,s)|} \bigg(\int_{B(0,2r)}|f(z)|\,dz\bigg)^2.\endaligned$$

Therefore, for any $f\in \mathscr{D}(D_{\rho, V})$ and $0<s\le r,$ by Jensen's inequality,
$$\aligned\int_{B(0,r)}&f^2(x)\,dx\\
\le & 2\int_{B(0,r)}\big(f(x)-f_s(x)\big)^2\,dx+ 2\int_{B(0,r)}f^2_s(x)\,dx\\
\le &2\int_{B(0,r)}\frac{1}{|B(0,s)|}\int_{B(x,s)}(f(x)-f(y))^2\,dy\,dx+ \frac{2}{|B(0,s)|} \bigg(\int_{B(0,2r)}|f(z)|\,dz\bigg)^2\\
\le & \bigg(\frac{2\sup_{0<\varepsilon\le s}\rho(\varepsilon)^{-1}}{|B(0,s)|}\bigg)\int_{B(0,r)}\int_{B(x,s)}(f(x)-f(y))^2\rho(|x-y|)\,dy\,dx\\
&+ \frac{2}{|B(0,s)|} \bigg(\int_{B(0,2r)}|f(z)|\,dz\bigg)^2\\
\le & \bigg(\frac{2\sup_{0<\varepsilon\le s}\rho(\varepsilon)^{-1}}{|B(0,s)|}\bigg)\int_{B(0,2r)}\int_{B(0,2r)}(f(x)-f(y))^2\rho(|x-y|)\,dy\,dx\\
&+ \frac{2}{|B(0,s)|} \bigg(\int_{B(0,2r)}|f(z)|\,dz\bigg)^2.\endaligned$$

(2) According to the inequality above, for any $f\in \mathscr{D}(D_{\rho, V})$ and $0<s\le r,$
\begin{align*}\int_{B(0,r)}&f^2(x)\,\mu_V(dx)\\
\le & \frac{1}{\inf_{|z|\le r} e^{V(z)}}\int_{B(0,r)}f^2(x)\,dx\\
\le & \bigg(\frac{2\big(\sup_{0<\varepsilon\le s}\rho(\varepsilon)^{-1}\big)}{|B(0,s)|\big(\inf_{|z|\le r} e^{V(z)}\big)}\bigg)\int_{B(0,2r)}\int_{B(0,2r)}(f(x)-f(y))^2\rho(|x-y|)\,dy\,dx\\
&+ \frac{2}{|B(0,s)|\big(\inf_{|z|\le r} e^{V(z)}\big)} \bigg(\int_{B(0,2r)}|f(z)|\,dz\bigg)^2\\
\le & \bigg(\frac{2\big(\sup_{0<\varepsilon\le s}\rho(\varepsilon)^{-1}\big)\big(\sup_{|z|\le 2r}e^{V(z)}\big)}{|B(0,s)|\big(\inf_{|z|\le r} e^{V(z)}\big)}\bigg)\\
&\qquad\qquad \times \int_{B(0,2r)}\int_{B(0,2r)}(f(x)-f(y))^2\rho(|x-y|)\,dy\,\mu_V(dx)\\
&+ \frac{2\big(\sup_{|z|\le 2r}e^{V(z)}\big)^2}{|B(0,s)|\big(\inf_{|z|\le r} e^{V(z)}\big)} \bigg(\int_{B(0,2r)}|f(x)|\,\mu_V(dx)\bigg)^2\\
\le & \bigg(\frac{2\big(\sup_{0<\varepsilon\le s}\rho(\varepsilon)^{-1}\big)\big(\sup_{|z|\le 2r}e^{V(z)}\big)}{|B(0,s)|\big(\inf_{|z|\le r} e^{V(z)}\big)}\bigg)D_{\rho, V}(f,f)\\
&+ \frac{2\big(\sup_{|z|\le 2r}e^{V(z)}\big)^2}{|B(0,s)|\big(\inf_{|z|\le r} e^{V(z)}\big)} \mu_V(|f|)^2.\end{align*}
The desired assertion for the case $0<s\le r$ follows from the conclusion above and the definition of $\beta_r$.

(3) When $s>r$, by (2),
$$ \int_{B(0,r)}f^2(x)\,\mu_V(dx)\le r D_{\rho, V}(f,f)+\beta_r(r) \mu_V(|f|)^2\le s D_{\rho, V}(f,f)+\beta_r(r\wedge s) \mu_V(|f|)^2.$$ The proof is completed.
 \end{proof}

We present the following two remarks for the proof of Theorem \ref{th1.1}.

\begin{itemize}
\item[(1)] The proof above is efficient for the following more general non-local Dirichlet form
$$\aligned \widetilde{D}_{j,V}(f,f):&=\iint_{x\neq y} \big(f(x)-f(y)\big)^2j(x,y)\,\mu_V(dy)\,\mu_V(dx),\\
\mathscr{D}(\widetilde{D}_{j,V}):&=\left\{f\in L^2(\mu_V):\widetilde{D}_{j,V}(f,f)<\infty\right\},\endaligned$$
where $j$ is a Borel measurable function
on $\R^{2d} \setminus\{(x, y) \in \R^{2d}: x = y\}$ such that $j(x, y) >0$ and $j(x, y) = j(y, x)$.
 See \cite[Section 2]{CW} for details.

\item[(2)] The argument above also works for $L^p$ $(p>1)$ setting. For instance, it can yield the statement as follows. If \eqref{th1.1} holds, then the following $L^p$-Poincar\'{e} inequality
    $$\mu_V(|f-\mu_V(f)|^p)\le 2c^{-1} \iint_{x\neq y} \frac{\big|f(x)-f(y)\big|^p}{|x-y|^{d+\alpha}}\,dy\,\mu_V(dx)=:2c^{-1}D_{\rho,V,p}(f,f)$$ holds, where
     $$ f\in \mathscr{D}(D_{\rho,V,p}):=\bigg\{f\in L^p(\mu_V): D_{\rho,V,p}(f,f)<\infty\bigg\},$$ and $L^p(\mu_V)$  denotes the set of Borel measurable functions $f$ on $\R^d$ such that $\int |f|^p(x)\,\mu_V(dx)<\infty.$ The proof is based on the argument of Theorem \ref{th1} (1) and the fact that for any $f\in L^p(\mu_V)$,
      $$\mu_V(|f-\mu_V(f)|^p)\le \iint |f(x)-f(y)|^p\,\mu_V(dy)\,\mu_V(dx),$$ due to the H\"{o}lder inequality. The readers can refer to \cite{HZ} for related discussion about $L^p$-Poincar\'{e} inequalities of local Dirichlet forms.
\end{itemize}

Now, we are in a position to give the

\begin{proof}[Proof of Theorem $\ref{th2}$]
(a) For any $f\in \mathscr{D}(D_{\rho, V})$ with $f>0$, by the Jensen inequality,
\begin{equation}\label{ent}
\begin{split}  \textrm{Ent}_{\mu_V}(f)&=\mu_V(f\log f)-\mu_V(f)\log \mu_V(f)\\
&\le \mu_V(f\log f)-\mu_V(f)\mu_V(\log f)\\
 &= \frac{1}{2}\iint \bigg[f(x)\log f(x)+f(y)\log f(y)\\
&\qquad\qquad-f(x)\log f(y)- f(y)\log f(x)\bigg]\,\mu_V(dy)\,\mu_V(dx)\\
&= \frac{1}{2}\iint (f(x)-f(y))(\log f(x)-\log f(y))\,\mu_V(dy)\,\mu_V(dx).\end{split}\end{equation}

Next, following the argument of Theorem 1.1 (1), we can obtain that under \eqref{th1.1}, for any $f\in \mathscr{D}(D_{\rho, V})$ with $f>0$,
 $$\aligned &\frac{1}{2}\iint (f(x)-f(y))(\log f(x)-\log f(y))\,\mu_V(dy)\,\mu_V(dx)\le c^{-1} D_{\rho,V}(f,\log f),
\endaligned$$ which, along with \eqref{ent}, completes the proof of the inequality \eqref{th2.2}.

(b) For any $p\in (1,2]$, $f\in \mathscr{D}(D_{\rho, V})$ with $f\ge0$, by the H\"{o}lder inequality,
\begin{equation} \label{beckner} \begin{split}\mu_V(f^p)&-\mu_V(f)^p\\
 \le&\mu_V(f^p)-\mu_V(f)\mu_V(f^{p-1})\\
=&\frac{1}{2}\iint \bigg[f^p(x)+f^p(y)-f(x)f^{p-1}(y)- f(y)f^{p-1}(x)\bigg]\,\mu_V(dy)\,\mu_V(dx)\\
=&\frac{1}{2}\iint (f(x)-f(y))(f^{p-1}(x)-f^{p-1}(y))\,\mu_V(dy)\,\mu_V(dx).\end{split}\end{equation}

Therefore, the desired Beckner-type inequality \eqref{th2.3} follows from \eqref{beckner} and the following fact
 $$\aligned &\frac{1}{2}\iint (f(x)-f(y))(f^{p-1}(x)-f^{p-1}(y))\,\mu_V(dy)\,\mu_V(dx)\le c^{-1} D_{\rho,V}(f, f^{p-1}),
\endaligned$$ where we have used \eqref{th1.1} again.
\end{proof}

To close this section, we present
\begin{proof}[Sketch of the Proof of Example $\ref{ex1}$] In this setting, $e^{-V(x)}=C_{d,\varepsilon}(1+|x|)^{-(d+\varepsilon)}$ and $\rho(r)=r^{-d-\alpha}.$
By the $C_r$-inequality, for any $x$, $y\in\R^d$ and $\varepsilon>0$,
$$|x-y|^{d+\varepsilon}\le 2^{d+\varepsilon-1}(|x|^{d+\varepsilon}+|y|^{d+\varepsilon})\le 2^{d+\varepsilon-1}\big((1+|x|)^{d+\varepsilon}+(1+|y|)^{d+\varepsilon}\big).$$

(a) For any $\varepsilon\ge \alpha$,
$$({e^{V(x)}+e^{V(y)}})\rho({|x-y|})\ge \frac{C_{d,\varepsilon}^{-1}\big((1+|x|)^{d+\varepsilon}+(1+|y|)^{d+\varepsilon}\big)}{2^{d+\alpha-1}\big((1+|x|)^{d+\alpha}+(1+|y|)^{d+\alpha}\big)}\ge \frac{2^{1-(d+\alpha)}}{C_{d,\varepsilon}}.$$ Combining it with Theorem \ref{th1} (1), we get the first assertion.

(b) For any $\varepsilon<\alpha$,
$$ \inf\limits_{0<|x-y|\le s}\big[(e^{V(x)}+e^{V(y)})\rho(|x-y|)\big]\ge C_{d,\varepsilon}^{-1}2^{1-(d+\varepsilon)}s^{\varepsilon-\alpha}.$$ Then, choosing $s= c r^{-1/\varepsilon}$ in the definition of $\alpha$, we arrive at the second assertion.

(c) For any $\varepsilon>\alpha$, \eqref{th1.2} holds with $\omega(x)=c_1(1+|x|)^{\varepsilon-\alpha}$, and $$\beta_t(s)\le c_2(1+s^{-d/\alpha}t^{(d+\varepsilon)(2+d/\alpha)}).$$ Then, the third assertion follows from the definition of $\beta$ by taking $s=c_3r$ and $t=c_4r^{-1/(\varepsilon-\alpha)}$. \end{proof}
\section{Applications: Porous media equations}
Functional inequalities for non-local Dirichlet forms appear throughout the probability literature, and also are interesting in analysis, e.g.\ see references in \cite{MRS,Gre}. This section is mainly motivated by \cite{DGGW, Wang} for the description of the convergence rate of porous media equations by using $L^p$ functional inequalities. Let $(L_{\rho,V},\mathscr{D}(L_{\rho,V}))$ be the generator corresponding to Dirichlet form $(D_{\rho,V}, \mathscr{D}(D_{\rho,V}))$. Consider the following equation
\begin{equation}\label{appl1}\partial_tu(t, \cdot)=L_{\rho,V}\{u(t, \cdot)^m\},\quad u(0,\cdot)=f,\end{equation} where $m>1$, $f$ is a bounded measurable function on $\R^d$ and $u^m:=\textrm{sgn}(u)|u|^m.$ We call $T_tf:=u(t,\cdot)$ a solution to the equation \eqref{appl1}, if $u(t,\cdot)^m\in \mathscr{D}(L_{\rho,V})$
for all $t>0$ and $u^m\in L^1_{\textrm{loc}}([0,\infty)\to \mathscr{D}(D_{\rho,V}); dt)$ such that, for any $g\in\mathscr{D}(D_{\rho,V}),$
$$\mu_V(u(t,\cdot)g)=\mu_V(fg)-\int_0^tD_{\rho,V}(u(s,\cdot)^m,g)\,ds,\quad t>0.$$
\begin{theorem}
Assume that for any bounded measurable function $f\in\mathscr{D}(L_{\rho,V})$ the equation \eqref{appl1} has a unique solution $T_tf$. If \eqref{th1.1} holds, then
$$\mu_V((T_tf)^2)\le \bigg[\mu_V(f^2)^{-(m-1)/2}+{c^{-1}(m-1)t}\bigg]^{-2/(m-1)},\quad t\ge0, \,\,\mu_V(f)=0.$$ \end{theorem}

\begin{proof} The argument of Theorem \ref{th2} gives us that, under \eqref{th1.1} for any $m>1$ and $f\in \mathscr{D}(L_{\rho,V})$ with $\mu_V(f)=0$,
\begin{equation}\label{proof1th4} \mu_V(f^{m+1})\le c^{-1} D_{\rho,V}(f,f^{m}). \end{equation}

Now, let $f$ be a function such that $\mu_V(f)=0$. Then, by the definition of the solution to the equation \eqref{appl1}, $\mu_V(T_tf)=0$ for all $t\ge0$. According to \eqref{proof1th4}, we obtain that
 $$\aligned \frac{d \mu_V(T_tf)^2}{dt}&=2\mu_V\Big(T_tf \partial_t T_tf\Big)=2\mu_V(T_tf L\{(T_tf)^m\})\\
 &=-2D_{\rho,V}(T_tf, (T_tf)^m)\le -2c^{-1}\mu_V((T_tf)^{m+1})\\
 &\le -2c^{-1}\Big[ \mu_V((T_tf)^2) \Big]^{\frac{m+1}{2}},
\endaligned$$ where in the inequality we have used the H\"{o}lder inequality. The required assertion easily follows from the inequality above.
 \end{proof}

\

\noindent{\bf Acknowledgements.} The author would like to thank two referees,
Professor Feng-Yu Wang and Dr.\ Xin Chen for helpful comments.
Financial support through National Natural Science Foundation of
China (No.\ 11201073) and the Program for New Century
Excellent Talents in Universities of Fujian (No.\ JA11051 and
JA12053) is also gratefully acknowledged.

\end{document}